\documentclass[reqno]{amsart}
\usepackage[dvipsnames]{xcolor}
\usepackage{amssymb,amscd,tikz,bbm,verbatim,graphicx,mathtools,enumitem}
\usepackage{textcomp}
\usepackage[mathscr]{euscript}
\usepackage[colorlinks]{hyperref}
\newtheorem{thm}{Theorem}[section]
\newtheorem{cor}[thm]{Corollary}

\theoremstyle{definition}

\newtheorem{remark}[thm]{Remark}
\newtheorem{hyp}[thm]{Hypothesis}

\numberwithin{equation}{section}

\begin{document}

\title[]{Mean Value Theorems and L'Hospital-Type Rules for Regulated Functions}%
\author{ahmed ghatasheh}

\address{Department of Mathematics, Philadelphia University, P.O.Box: 1 Amman - Jordan 19392}
\email{aghatasheh@philadelphia.edu.jo}
\email{med.ghatasheh@gmail.com}


\begin{abstract}
We introduce a generalization of Cauchy’s mean value theorem for regulated functions. Building on this, we extend both L’Hospital’s rule and L’Hospital’s monotone rule to quotients of regulated functions. We demonstrate that our extended L’Hospital’s rule encompasses both the discrete case, known as the Stolz-Cesaro theorem, and the classical continuous case. In addition, we show that these extensions handle some problems that classical rules cannot address. Finally, we provide Lebesgue-Stieltjes versions of L’Hospital’s rule and L’Hospital’s monotone rule and compare them with our extensions.
\end{abstract}
\maketitle

\section{Introduction}

The classical L'Hospital's rule, which dates back to 1696, states that
\begin{equation}\label{LHR00001}
\lim_{x\uparrow b}\frac{f(x)}{g(x)}=\lim_{x\uparrow b}\frac{f'(x)}{g'(x)}
\end{equation}
provided that
\begin{enumerate}[leftmargin=.65cm]
\item the functions $f$ and $g$ are differentiable everywhere on an interval $(a,b)$, where $-\infty\leq a<b\leq\infty$,
\item $g'$ does not vanish anywhere in $(a,b)$,
\item $f'(x)/g'(x)$ tends to $A\in[-\infty,\infty]$ as $x\uparrow b$, and
\item either $f(x)$ and $g(x)$ tend to zero as $x\uparrow b$, $g(x)$ tends to $\infty$ as $x\uparrow b$, or $g(x)$ tends to $-\infty$ as $x\uparrow b$.
\end{enumerate}
On the other hand L'Hospital's monotone rule states that the function $f/g$ is increasing on the interval $(a,b)$ provided that the first and the second statements hold, $f'/g'$ is increasing on $(a,b)$, and either $\lim_{x\downarrow a}f(x)=\lim_{x\downarrow a}g(x)=0$ or $\lim_{x\uparrow b}f(x)=\lim_{x\uparrow b}g(x)=0$. For the proof of these rules see Rudin \cite{rudin1976principles}, Anderson et al. \cite{MR1227499}, and Pinelis \cite{MR1888920}.

One of the objectives of this article is to generalize these L'Hospital-type rules when the functions $f$ and $g$ are regulated on $(a,b)$. The function $f$ is called regulated on $(a,b)$ if for each $x\in(a,b)$ the left-hand limit $f^-(x)=\lim_{t\uparrow x}f(t)$ and the right-hand limit $f^+(x)=\lim_{t\downarrow x}f(t)$ exist.

Our generalizations  of these L'Hospital-type rules are given by Theorems \ref{LHR00017} and \ref{LHR00023}. The differentiation used in these generalizations differs slightly from the classical one; it is defined with respect to a strictly increasing function $\alpha$ on $(a,b)$ via 
\begin{equation}\label{LHR00002}
\left(D_{\alpha}f\right)(x)=\lim_{h\downarrow 0}\frac{f^-(x+h)-f^+(x-h)}{\alpha^-(x+h)-\alpha^+(x-h)}
\end{equation}
provided the limit exists. These generalizations show that the statements of these rules hold with respect to the differentiation defined by \eqref{LHR00002} under certain conditions that we impose on $f$, $g$, and $\alpha$.  

If $f'(x)$ exists for all $x\in(a,b)$, then picking $\alpha(x)=x$ gives $D_{\alpha}f=f'$. This means \eqref{LHR00002} generalizes the classical derivative. We will provide some properties of regulated functions and such differentiation in the next section.

The standard proofs of the classical L'Hospital-type rules rely on mean value theorems, in particular Cauchy's mean value theorem, see Theorem 4.9 in Rudin \cite{rudin1976principles}. Other interesting proofs of L'Hospital's rule can be found in Taylor \cite{10.2307/2307183}. In Section \ref{LHR00005}, we extend Cauchy's mean value theorem to regulated functions, with respect to the differentiation  defined by \eqref{LHR00002}. In Section \ref{LHR00016.5}, we use this extension to prove our generalized versions of these rules. There are many generalizations of mean value theorems; for interested readers we mention \cite{Matkowski+2010+765+774}, \cite{balogh2016functional}, \cite{Mateljevic2013GeneralizationsOT}, \cite{lozada2020some}, \cite{10.2307/24896380}, \cite{Witula2012MeanvalueTF}, and the references therein. In addition, the book \cite{doi:10.1142/3857} is an excellent reference for mean value theorems. Most of these generalizations assume the continuity of the underlying functions, which is not our case. 

We show that our generalizations can handle examples that cannot be addressed by some of the well-known versions of L'Hospital's rule, see Subsections \ref{LHR00025} and \ref{LHR00026}. Also, we show that the Stolz-Cesaro theorem is a special case of our generalization, see Subsection \ref{LHR00027}. Then we provide an example that involves a quotient of two functions that have discontinuities, see Subsection \ref{LHR00028}.

The second requirement in L'Hospital's rule \eqref{LHR00001} is equivalent to that either $g'$ is positive on $(a,b)$ or $g'$ is negative on $(a,b)$, because of the intermediate value property of derivatives. This requirement implies that the function $g$ is strictly monotone on $(a,b)$. Such requirement is essential, see the family of counterexamples provided by Nester and Vyborny \cite{Vyborny1989}. Another family of examples was provided by Boas \cite{doi:10.1080/00029890.1986.11971912}. However, if the function $g$ is strictly monotone and fails to satisfy the second requirement, then L'Hospital's rule may be applicable, see the example provided by Gorni \cite{10.2307/2323839}. The remaining requirements were slightly relaxed by some authors. For example, in \cite{Vyborny1989} authors relaxed the first requirement by providing a version of L'Hospital's rule for one-sided derivatives provided that the functions $f$ and $g$ are continuous. Another important generalization via essential limits is due to Lee \cite{10.2307/2040953}. Other generalizations that rely on absolute continuity of the functions $f$ and $g$ are due to Ostrowski \cite{10.2307/2318210} and Vianello \cite{10.2307/44152304}.

There are some other generalizations of L'Hospital's rule; for example a version for multivariable functions was provided by Ivele and Shilin \cite{Ivlev2014OnAG} and another one by Lawlor \cite{doi:10.1080/00029890.2020.1793635}. In addition, a version for complex-valued functions was provided by Carter \cite{10.2307/2310244}.

In Section \ref{LHR00029}, we introduce a measure-theoretic approach to L'Hospital's rule and L'Hospital's monotone rule via Lebesgue-Stieltjes integrals. Finally, in Section \ref{LHR00037.5}, we give a comparison between the earlier versions of L'Hospital-type rules, i.e., Theorems \ref{LHR00017} and \ref{LHR00023}, and the measure-theoretic approach in Section \ref{LHR00029}. These two sections are motivated by Estrada and Pavlovic \cite{Estrada2017LHpitalsMR}; who provided a comparison for theses rules in the classical case.

\section{Preliminaries}\label{qw682jhc42l,k529b}

We begin by introducing some important properties of regulated functions. Recall that a function $f:(a,b)\to\mathbb{R}$ is called regulated on $(a,b)$ if $f^-(x)=\lim_{t\uparrow x}f(t)$ and $f^+(x)=\lim_{t\downarrow x}f(t)$ exist for all $x\in(a,b)$. Sums and products of regulated functions are again regulated functions. Additionally, every regulated function has a countable number of discontinuities, see Lemma 1.4.1 in Lakshmikantham \cite{lakshmikantham2017monotone}. Also, functions of locally bounded variation on $(a,b)$ are regulated on $(a,b)$.

We will use the following result frequently in this paper. We state it without proof as the reasoning is straightforward.

\begin{thm}\label{LHR00003}
Suppose that $f$ is regulated on $(a,b)$. Then $f^-$ and $f^+$ are regulated on $(a,b)$ and
\begin{equation}\label{LHR00004}
{(f^-)^-=(f^+)^-=f^-}\text{  and  }(f^+)^+=(f^-)^+=f^+.
\end{equation}
Furthermore, the following statements are equivalent:
\begin{enumerate}[leftmargin=.65cm]
\item $f^-$ is strictly increasing on $(a,b)$.
\item $f^+$ is strictly increasing on $(a,b)$.
\item $f^+(s)<f^-(t)$ for all $a<s<t<b$.
\end{enumerate}
\end{thm}


Definition \eqref{LHR00002} is adopted from measure theory, see for example Evans and Gariepy \cite{MR3409135} for more details. Also, see the introduction of Section \ref{LHR00029} and Theorem \ref{LHR00029.5}. This definition allows us to differentiate at a point $x$ even if $f$ and $\alpha$ are not continuous at $x$, which gives the differentiation process more flexibility. Using equations \eqref{LHR00004}, it is easy to verify that
\begin{equation}\label{LHR00004.5}
\left(D_{\alpha}f\right)(x)=\lim_{h\downarrow 0}\frac{f^+(x+h)-f^-(x-h)}{\alpha^+(x+h)-\alpha^-(x-h)}.
\end{equation}

The following are some properties that follow directly from Definition \eqref{LHR00002}. We assume that $-\infty\leq a<b\leq\infty$, that the functions $f$ and $g$ are regulated on $(a,b)$, and that $\alpha$ is strictly increasing on $(a,b)$, unless otherwise stated.

\begin{enumerate}[leftmargin=.65cm]
\item $D_{\alpha}f$, $D_{\alpha}f^-$, and $D_{\alpha}f^+$ are identical. This follows from Equations \eqref{LHR00004}. 
\item $\alpha^-(x)=\alpha^+(x)$ at some point $x\in(a,b)$ implies that $f^-(x)=f^+(x)$ provided that $\left(D_{\alpha}f\right)(x)$ exists.
\item If $\alpha^-(x)\neq\alpha^+(x)$ at some point $x\in(a,b)$, then
$$\left(D_{\alpha}f\right)(x)=\frac{f^+(x)-f^-(x)}{\alpha^+(x)-\alpha^-(x)}.$$
\item Suppose that $\alpha(t)=t$. If the classical one-sided derivatives $f'_-(x)$ and $f'_+(x)$ exist at each $x\in(a,b)$, then $D_{\alpha}f$ averages $f'_-$ and $f'_+$, i.e., $D_{\alpha}f=(f'_-+f'_+)/2$.
\item  If $\left(D_{\alpha}f\right)(x)$ exists at some point $x\in(a,b)$, it does not guarantee the existence of $f'_-(x)$, $f'_+(x)$, or $f'(x)$. For example if $\alpha(t)=t$ and $f(t)=\vert t\vert$, then $\left(D_{\alpha}f\right)(0)=0$ but $f'(0)$ does not exist. Also, if 
$$g(t)=\begin{cases}
t\sin(1/t), & t>0\\
0, & t=0\\
t\sin(1/t)+2t, & t<0,
\end{cases}$$
then $\left(D_{\alpha}g\right)(0)=1$, but neither $g'_-(0)$ nor $g'_+(0)$ exists.
\item If $f$ and $\alpha$ are differentiable everywhere on $(a,b)$, then
$\left(D_{\alpha}f\right)(x)=f'(x)/\alpha'(x)$ provided that $\alpha'(x)\neq0$. However, It is easy to construct a strictly increasing function $\alpha$ on $(a,b)$ for which $\alpha'(x)$ exists everywhere and that $\alpha'(x)=0$ on a set of positive Lebesgue measure. 

\item $D_{\alpha}f$ does not have the intermediate value property. For example if $f(t)=\vert t\vert$ and $\alpha(t)=t$, then there is no $x\in(-1,1)$ that satisfies $\left(D_{\alpha}f\right)(x)=1/2$ although $1/2$ is between $\left(D_{\alpha}f\right)(1)=1$ and $\left(D_{\alpha}f\right)(-1)=-1$.
\item If $\left(D_{\alpha}f\right)(x)$ and $\left(D_{\alpha}g\right)(x)$ exist at some point $x\in(a,b)$, then the product rule is given by
$$
\begin{aligned}
\left(D_{\alpha}fg\right)(x)&=g^-(x)\left(D_{\alpha}f\right)(x)+f^+(x)\left(D_{\alpha}g\right)(x)\\
&=g^+(x)\left(D_{\alpha}f\right)(x)+f^-(x)\left(D_{\alpha}g\right)(x).
\end{aligned}
$$
In addition, the quotient rule is given by
$$
\begin{aligned}
\left(D_{\alpha}\frac{f}{g}\right)(x)&=\frac{g^-(x)\left(D_{\alpha}f\right)(x)-f^-(x)\left(D_{\alpha}g\right)(x)}{g^-(x)g^+(x)}\\
&=\frac{g^+(x)\left(D_{\alpha}f\right)(x)-f^+(x)\left(D_{\alpha}g\right)(x)}{g^-(x)g^+(x)},
\end{aligned}
$$
provided that $g^-(x)g^+(x)\neq0$.
\end{enumerate}

\section{Mean Value Theorems}\label{LHR00005}

The goal of this section is to introduce an equivalent of Cauchy's mean value theorem that aligns with Definition \eqref{LHR00002}. We will use differentiation techniques for the proofs, although some results could also be established using Lebesgue-Stieltjes integrals.

Consider $f(x)=-2x$ when $x\leq 0$ and $f(x)=4x$ when $x>0$. Then $f(-2)=f(1)=4$, but with $\alpha(x)=x$, we get
$$\left(D_{\alpha}f\right)(x)=\begin{cases}
-2, & x<0\\
1, & x=0\\
4, & x>0.
\end{cases}$$
Thus $D_{\alpha}f$ does not vanish anywhere in $(-2,1)$, which means Rolle's theorem fails in this case. However, there are $u$ and $v$ in $(-2,1)$ such that $\left(D_{\alpha}f\right)(u)\left(D_{\alpha}f\right)(v)\leq 0$. We will see later on that this is true in general.

Before we begin, let us outline the assumptions needed for most of the upcoming results in the following hypothesis.

\begin{hyp}\label{LHR00006}\
\begin{enumerate}[leftmargin=.65cm]
\item $-\infty\leq a<b\leq\infty$.
\item $f$ is a regulated real-valued function on $(a,b)$.
\item $g$ is a regulated real-valued function on $(a,b)$.
\item $\alpha$ is a strictly increasing function on $(a,b)$.
\item $\left(D_{\alpha}f\right)(x)$ exist for all $x\in(a,b)$.
\item $\left(D_{\alpha}g\right)(x)$ exist for all $x\in(a,b)$.
\end{enumerate}
\end{hyp}

The following fact is the main tool needed to develop mean value theorems for regulated functions.

\begin{thm}\label{LHR00007}
Suppose that $f$ and $\alpha$ satisfy assumptions $1$, $2$, $4$, and $5$ in Hypothesis \ref{LHR00006}. The following statements hold:
\begin{enumerate}[leftmargin=.65cm]
\item If $D_{\alpha}f>0$ on $(a,b)$, then $f^+(x)<f^-(y)$ for all $x<y$ in $(a,b)$.
\item If $D_{\alpha}f\geq0$ on $(a,b)$, then $f^+(x)\leq f^-(y)$ for all $x<y$ in $(a,b)$.
\end{enumerate}
\end{thm}

\begin{proof}
It follows from the positivity of $D_{\alpha}f$ on $(a,b)$ that for each $t\in(a,b)$ there exists $\epsilon_t>0$ (depending on $t$) such that
\begin{equation}\label{LHR00008}
f^+(t-h)\leq f^-(t+h)
\end{equation}
for all $h\in(0,\epsilon_t)$. Letting $h\downarrow0$ gives
\begin{equation}\label{LHR00009}
f^-(t)\leq f^+(t)
\end{equation}
for all $t\in(a,b)$.

Let $x<y$ in $(a,b)$ be given and set $z=(x+y)/2$. Consider the set
$$S=\left\{\epsilon\in(0,y-z]:f^+(z-h)\leq f^-(z+h) \text{ for all }h\in(0,\epsilon]\right\}.$$
Note that $S$ is not empty since inequality \eqref{LHR00008} is valid when $t=z$. Let $s$ be the least upper bound of $S$. It is straightforward to verify that $s\in S$. We claim that $s=y-z$, which implies that $f^+(x)\leq f^-(y)$. To prove our claim assume to the contrary that $s<y-z$. Since $\left(D_{\alpha}f\right)(z-s)>0$ and $\left(D_{\alpha}f\right)(z+s)>0$, there exists $\gamma\in(0,\min\{y-z-s,s\})$ such that
\begin{equation}\label{LHR00010}
f^+(z-s-h)\leq f^-(z-s+h)
\end{equation}
and
\begin{equation}\label{LHR00011}
f^+(z+s-h)\leq f^-(z+s+h)
\end{equation}
for all $h\in(0,\gamma)$. If $r\in(s,s+\gamma)$ (see the figure below), then

$$
\begin{tikzpicture}

  \coordinate (x) at (0,0);
  
  \coordinate (z) at (6,0);

  \coordinate (y) at (12,0);
 
  \coordinate (z+s) at (9.5,0);
  
  \coordinate (z-s) at (2.5,0);
  
  \coordinate (z-r) at (1.5,0);
  
  \coordinate (z+r) at (10.5,0);
  
  \coordinate (z-2s+r) at (3.5,0);
  \coordinate (z+2s-r) at (8.5,0);  
  
  \draw[semithick] (x) -- (y);

  \filldraw (x) circle (1.5pt) node[below] {$x$};
  \filldraw (y) circle (1.5pt) node[below] {$y$};
  \filldraw (z) circle (1.5pt) node[below] {$z$};
  
  \filldraw (z+s) circle (1.5pt) node[below] {$z+s$};
  \filldraw (z-s) circle (1.5pt) node[below] {$z-s$};
    
  \filldraw (z-r) circle (1.5pt) node[above] {$z-r$};
  \filldraw (z+r) circle (1.5pt) node[above] {$z+r$};
  
  \filldraw (z-2s+r) circle (1.5pt) node[above] {$z-2s+r$}; 
   \filldraw (z+2s-r) circle (1.5pt) node[above] {$z+2s-r$};  
\end{tikzpicture}
$$ 

$$f^+(z-r)=f^+(z-s-(r-s))\leq f^-(z-s+(r-s))=f^-(z-(2s-r)),$$
where the inequality is justified by the validity of \eqref{LHR00010} when $h=r-s$. Inequality \eqref{LHR00009} informs us that $f^-(z-(2s-r))\leq f^+(z-(2s-r))$. But $2s-r\in(0,s)$, so
$$f^+(z-(2s-r))\leq f^-(z+(2s-r))=f^-(z+s-(r-s))$$
Again inequality \eqref{LHR00009} informs us that $f^-(z+s-(r-s))\leq f^+(z+s-(r-s))$. Finally, since $r-s\in(0,\gamma)$ it follows from inequality \eqref{LHR00011} that
$$f^+(z+s-(r-s))\leq f^-(z+s+(r-s))=f^-(z+r).$$
We conclude that $f^+(z-r)\leq f^-(z+r)$ for all $r\in(0,s+\epsilon_0)$, which contradicts $s$ being the least upper bound of $S$. This completes the proof of our claim.

We have shown so far that $f^+(x)\leq f^-(y)$ for all $x<y$ in $(a,b)$. If $f^+(x_0)=f^-(y_0)$ for some $x_0<y_0$ in $(a,b)$, then 
$$f^+(x_0)\leq f^-(t)\leq f^+(t)\leq f^-(y_0)$$ 
for all $t\in(x_0,y_0)$. This means $f^-$ is constant on $(x_0,y_0)$, implying that $D_{\alpha}f$ is identically zero on $(x_0,y_0)$. Hence $f^+(x)<f^-(y)$ for all $x<y$ in $(a,b)$, which completes the proof of the first statement.

To prove the second statement assume that $D_{\alpha}f\geq0$ on $(a,b)$ and let $x<y$ in $(a,b)$ be given. For each $\epsilon>0$ consider the function $\phi_{\epsilon}=\epsilon\alpha+f$. Obviously $D_{\alpha}\phi_{\epsilon}=\epsilon+D_{\alpha}f>0$. It follows from the first statement that
$$\epsilon\alpha^+(x)+f^+(x)<\epsilon\alpha^-(y)+f^-(y).$$
Letting $\epsilon\downarrow0$ gives $f^+(x)\leq f^-(y)$, which completes the proof.
\end{proof}

\begin{remark}
The first claim in Theorem \ref{LHR00007} is true if $\left(D_{\alpha}f\right)(x)\in(0,\infty]$ for all $x\in(a,b)$.  Similarly, the second claim is true if $\left(D_{\alpha}f\right)(x)\in[0,\infty]$.
\end{remark}

Now we are ready to present a generalization of Rolle's theorem that applies to regulated functions.

\begin{thm}[Rolle's Theorem for Regulated Functions]\label{LHR00012}
Suppose that $f$ and $\alpha$ satisfy assumptions $1$, $2$, $4$, and $5$ in Hypothesis \ref{LHR00006}. If $f^+(s)=f^-(t)$ for some $s<t$ in $(a,b)$, then there are $u$ and $v$ in $(s,t)$ such that
$\left(D_{\alpha}f\right)(u)\left(D_{\alpha}f\right)(v)\leq 0$.
\end{thm}

\begin{proof}
If $\left(D_{\alpha}f\right)(u)\left(D_{\alpha}f\right)(v)>0$ for all $u$ and $v$ in $(s,t)$, then either $D_{\alpha}f$ is positive on $(s,t)$ or $D_{\alpha}f$ is negative on $(s,t)$, which leads to $f^+(s)\neq f^-(t)$ according to Theorem \ref{LHR00007}. This completes the proof.
\end{proof}

Theorem \ref{LHR00012} generalizes the well-known Rolle's Theorem. To see this, if we assume that $f$ is differentiable on $(s,t)$ in the classical sense and that $f$ is continuous on $[s,t]$, then $f'=D_{\alpha}f$ on $(s,t)$, where $\alpha(x)=x$. It follows from Theorem \ref{LHR00012} that $f'(u)f'(v)\leq 0$ for some $u$ and $v$ in $(s,t)$. If $f'(u)f'(v)=0$, then we are done. Otherwise, the intermediate value property of derivatives assures the existence of $w$ between $u$ and $v$ for which $f'(w)=0$.

\begin{thm}[Cauchy’s Mean Value Theorem for Regulated Functions]\label{LHR00013}
Suppose that $f$, $g$, and $\alpha$ satisfy Hypothesis \ref{LHR00006}. If $s<t$ are in $(a,b)$, then there are $u$ and $v$ in $(s,t)$ such that the product of
\begin{equation}\label{LHR00014}
\left(g^-(t)-g^+(s)\right)\left(D_{\alpha}f\right)(u)-\left(f^-(t)-f^+(s)\right)\left(D_{\alpha}g\right)(u)
\end{equation}
and
\begin{equation}\label{LHR00015}
\left(g^-(t)-g^+(s)\right)\left(D_{\alpha}f\right)(v)-\left(f^-(t)-f^+(s)\right)\left(D_{\alpha}g\right)(v)
\end{equation}
is less than or equal to zero.
\end{thm}

\begin{proof}
Define the regulated function $h$ on $(a,b)$ by
$$h(x)=\left(g^-(t)-g^+(s)\right)f(x)-\left(f^-(t)-f^+(s)\right)g(x).$$
Observe that $h^+(s)=h^-(t)=g^-(t)f^+(s)-g^+(s)f^-(t)$ and $\left(D_{\alpha}h\right)(x)$ exists for all $x\in(a,b)$. It follows from Theorem \ref{LHR00012} that $\left(D_{\alpha}h\right)(u)\left(D_{\alpha}h\right)(v)\leq 0$ for some $u$ and $v$ in $(s,t)$ which completes the proof.
\end{proof}

\begin{cor}\label{LHR00016}
Suppose that $f$, $g$, and $\alpha$ satisfy Hypothesis \ref{LHR00006}. In addition, suppose that either $D_{\alpha}g$ is positive on $(a,b)$ or $D_{\alpha}g$ is negative on $(a,b)$. The following statements hold:
\begin{enumerate}[leftmargin=.65cm]
\item If $s<t$ are in $(a,b)$, then there are $u$ and $v$ in $(s,t)$ such that
$$
\frac{\left(D_{\alpha}f\right)(u)}{\left(D_{\alpha}g\right)(u)}\leq\frac{f^-(t)-f^+(s)}{g^-(t)-g^+(s)}\leq\frac{\left(D_{\alpha}f\right)(v)}{\left(D_{\alpha}g\right)(v)}.$$
\item If $\left(D_{\alpha}f\right)/\left(D_{\alpha}g\right)$ is increasing on $(a,b)$, then for any $s<t$ in $(a,b)$ we have
$$\frac{\left(D_{\alpha}f\right)(s)}{\left(D_{\alpha}g\right)(s)}\leq\frac{f^-(t)-f^+(s)}{g^-(t)-g^+(s)}\leq\frac{\left(D_{\alpha}f\right)(t)}{\left(D_{\alpha}g\right)(t)}.$$
\end{enumerate}
\end{cor}

\begin{proof}
To prove the first statement note that Theorem \ref{LHR00013} assures the existence of $u$ and $v$ in $(s,t)$ such that the product of the expressions \ref{LHR00014} and \ref{LHR00015} is less than or equal to zero. Since $D_{\alpha}g$ does not vanish anywhere in $(a,b)$, we may assume that
$$
\frac{\left(D_{\alpha}f\right)(u)}{\left(D_{\alpha}g\right)(u)}\leq\frac{\left(D_{\alpha}f\right)(v)}{\left(D_{\alpha}g\right)(v)}.$$
Since $D_{\alpha}g$ is either positive or negative on $(a,b)$, it follows from Theorem \ref{LHR00007} that $\left(D_{\alpha}g\right)(u)$, $\left(D_{\alpha}g\right)(v)$, and $g^-(t)-g^+(s)$ have the same sign. Therefore 
$$\left(\frac{\left(D_{\alpha}f\right)(u)}{\left(D_{\alpha}g\right)(u)}-\frac{f^-(t)-f^+(s)}{g^-(t)-g^+(s)}\right)\left(\frac{\left(D_{\alpha}f\right)(v)}{\left(D_{\alpha}g\right)(v)}-\frac{f^-(t)-f^+(s)}{g^-(t)-g^+(s)}\right)\leq 0,$$
which completes the proof of the first statement.

Observe that since $\left(D_{\alpha}f\right)/\left(D_{\alpha}g\right)$ is increasing on $(a,b)$ we have
$$\frac{\left(D_{\alpha}f\right)(s)}{\left(D_{\alpha}g\right)(s)}\leq\frac{\left(D_{\alpha}f\right)(u)}{\left(D_{\alpha}g\right)(u)}\text{  and  }\frac{\left(D_{\alpha}f\right)(v)}{\left(D_{\alpha}g\right)(v)}\leq\frac{\left(D_{\alpha}f\right)(t)}{\left(D_{\alpha}g\right)(t)},$$
which completes the proof of the second statement.
\end{proof}

We will introduce a measure-theoretic application to the above mean value theorems in Section \ref{LHR00029}, see Theorem \ref{LHR00029.01}.

\section{L'Hospital-Type Rules}\label{LHR00016.5}

\begin{thm}[L'Hospital's Rule for Regulated Functions]\label{LHR00017}
Suppose that $f$, $g$, and $\alpha$ satisfy Hypothesis \ref{LHR00006}. In addition, suppose that $D_{\alpha}g$ is either positive or negative on $(a,b)$ and that
\begin{equation}\label{LHR00018}
\lim_{x\uparrow b}\frac{\left(D_{\alpha}f\right)(x)}{\left(D_{\alpha}g\right)(x)}=A,
\end{equation}
where $A\in[-\infty,\infty]$. If $\lim_{x\uparrow b}f^-(x)=\lim_{x\uparrow b}g^-(x)=0$, $\lim_{x\uparrow b}g^-(x)=-\infty$, or $\lim_{x\uparrow b}g^-(x)=\infty$, then
\begin{equation}\label{LHR00019}
\lim_{x\uparrow b}\frac{f^-(x)}{g^-(x)}=A.
\end{equation}
\end{thm}

\begin{proof}
For each $s<t$ in $[a,b]$ we define the set
\begin{equation}
R(s,t)=\left\{\frac{\left(D_{\alpha}f\right)(x)}{\left(D_{\alpha}g\right)(x)}:x\in(s,t)\right\}.
\end{equation}
We denote the infimum and the supremum of $R(s,t)$ by $\inf R(s,t)$ and $\sup R(s,t)$, respectively. Note that $\inf R(s,t)$ and $\sup R(s,t)$ belong to $[-\infty,\infty]$. The first statement in Corollary \ref{LHR00016} informs us that
\begin{equation}\label{LHR00020}
\inf R(s,t)\leq\frac{f^-(t)-f^+(s)}{g^-(t)-g^+(s)}\leq\sup R(s,t)
\end{equation}
for any $s<t$ in $(a,b)$. Since $\inf R(s,b)\leq\inf R(s,t)$ and $\sup R(s,t)\leq\sup R(s,b)$ for all $s<t$ in $(a,b)$ we have
\begin{equation}\label{LHR00021}
\inf R(s,b)\leq\frac{f^-(t)-f^+(s)}{g^-(t)-g^+(s)}\leq\sup R(s,b).
\end{equation}

Now assume that $\lim_{x\uparrow b}f^-(x)=\lim_{x\uparrow b}g^-(x)=0$. By letting $t\uparrow b$ in inequality \eqref{LHR00021} we obtain 
\begin{equation}\label{LHR00022}
\inf R(s,b)\leq\frac{f^+(s)}{g^+(s)}\leq\sup R(s,b).
\end{equation}
Then letting $s\uparrow b$ in inequality \eqref{LHR00022} proves the limit \eqref{LHR00019}.

Now assume that $\lim_{x\uparrow b}g^-(x)=\infty$. We may assume that $g^-$ is positive on $(a,b)$. Note that this forces $D_{\alpha}g$ to be positive in $(a,b)$ since it is either positive or negative on $(a,b)$. Therefore $g^+(s)<g^-(t)$ for all $s<t$ in $(a,b)$. We rewrite inequality \eqref{LHR00021} as
\begin{equation}
\left(1-\frac{g^+(s)}{g^-(t)}\right)\inf R(s,b)\leq\frac{f^-(t)}{g^-(t)}-\frac{f^+(s)}{g^-(t)}\leq\left(1-\frac{g^+(s)}{g^-(t)}\right)\sup R(s,b).
\end{equation}
By taking $\liminf$ and $\limsup$ as $t\uparrow b$ we obtain
$$\inf R(s,b)\leq\liminf_{t\uparrow b}\frac{f^-(t)}{g^-(t)}\leq\sup R(s,b)$$
and 
$$\inf R(s,b)\leq\limsup_{t\uparrow b}\frac{f^-(t)}{g^-(t)}\leq\sup R(s,b),$$
respectively. Then letting $s\uparrow b$ gives us \eqref{LHR00019} which completes the proof of the infinity case. The negative infinity case is just a variation of the infinity case, so this completes the proof of the theorem.
\end{proof}

\begin{remark}\
\begin{enumerate}[leftmargin=.65cm]
\item A similar version of Theorem \ref{LHR00017} works at the endpoint $a$.
\item Theorem \ref{LHR00017} is valid if the fourth assumption in Hypothesis \ref{LHR00006} is replaced by $\alpha$ is strictly decreasing on $(a,b)$.
\item Assumption \eqref{LHR00018} in Theorem \ref{LHR00018} can be written as $\lim_{x\uparrow b}\left(D_gf\right)(x)=A$. However, if $f$ and $g$ are functions for which $\lim_{x\uparrow b}\left(D_gf\right)(x)=A$, then it is not necessarily true that $\lim_{x\uparrow b}f^-(x)/g^-(x)=A$. For example, if we consider the functions $f(x)=1+h(x)$ and $g(x)=2+h(x)$, where $h(x)=\sin(\sqrt{x})$, then $\left(D_gf\right)(x)=1$ when $x>0$ but $\lim_{x\uparrow b}f(x)/g(x)$ does not exist. This tells us that the existence of the function $\alpha$ that satisfies the last three assumptions in Hypothesis \ref{LHR00006} is necessary for Theorem \ref{LHR00017}.
\item For any regulated function $\beta$ on $(a,b)$, $\beta^-(x)\rightarrow A$ as $x\uparrow b$ if and only if $\beta^+(x)\rightarrow A$ as $x\uparrow b$. In other words, $\beta^-$ and $\beta^+$ have the same end behavior as $x\uparrow b$. This means our conclusion \eqref{LHR00019} in Theorem \ref{LHR00017} is equivalent to $\lim_{x\uparrow b}f^+(x)/g^+(x)=A$.
\end{enumerate}
\end{remark}

\begin{thm}[L'Hospital's Monotone Rule]\label{LHR00023}
Suppose that $f$, $g$, and $\alpha$ satisfy Hypothesis \ref{LHR00006}, that $D_{\alpha}g$ is either positive or negative on $(a,b)$, and that either $\lim_{x\uparrow b}g^-(x)=\lim_{x\uparrow b}f^-(x)=0$ or $\lim_{x\downarrow a}g^-(x)=\lim_{x\downarrow a}f^-(x)=0$. The following statements hold:
\begin{enumerate}[leftmargin=.65cm]
\item If $\left(D_{\alpha}f\right)/\left(D_{\alpha}g\right)$ is increasing on $(a,b)$, then $f^-/g^-$ is increasing on $(a,b)$.
\item If $\left(D_{\alpha}f\right)/\left(D_{\alpha}g\right)$ is strictly increasing on $(a,b)$, then $f^-/g^-$ is strictly increasing on $(a,b)$.
\end{enumerate} 
\end{thm}

\begin{proof}
Since $D_{\alpha}g$ is either positive or negative on $(a,b)$, Theorems \ref{LHR00003} and \ref{LHR00007} inform us that either $g^-$ and $g^+$ are strictly increasing on $(a,b)$ or $g^-$ and $g^+$ are strictly decreasing on $(a,b)$. Hence the hypothesis that either $\lim_{x\uparrow b}g^-(x)=0$ or $\lim_{x\downarrow a}g^-(x)=0$ implies that either $g^-$ and $g^+$ are positive on $(a,b)$ or $g^-$ and $g^+$ are negative on $(a,b)$. We conclude that if $\lim_{x\uparrow b}g^-(x)=0$, then $\left(D_{\alpha}g\right)/g^+<0$ on $(a,b)$ and that if $\lim_{x\downarrow a}g^-(x)=0$, then $\left(D_{\alpha}g\right)/g^+>0$ on $(a,b)$.

Observe that the quotient rule can be written as
\begin{equation}
\begin{aligned}
\left(D_{\alpha}\frac{f}{g}\right)(x)
&
=\frac{\left(D_{\alpha}g\right)(x)}{g^+(x)}\left(\frac{\left(D_{\alpha}f\right)(x)}{\left(D_{\alpha}g\right)(x)}-\frac{f^-(x)}{g^-(x)}\right)\\
&
=\frac{\left(D_{\alpha}g\right)(x)}{g^-(x)}\left(\frac{\left(D_{\alpha}f\right)(x)}{\left(D_{\alpha}g\right)(x)}-\frac{f^+(x)}{g^+(x)}\right).
\end{aligned}
\end{equation}
Now suppose that $\left(D_{\alpha}f\right)/\left(D_{\alpha}g\right)$ is increasing on $(a,b)$. The second statement in Corollary \ref{LHR00016} informs us that for any $s<t$ in $(a,b)$ we have
$$\frac{\left(D_{\alpha}f\right)(s)}{\left(D_{\alpha}g\right)(s)}\leq\frac{f^-(t)-f^+(s)}{g^-(t)-g^+(s)}\leq\frac{\left(D_{\alpha}f\right)(t)}{\left(D_{\alpha}g\right)(t)}.$$
If $\lim_{x\uparrow b}g^-(x)=\lim_{x\uparrow b}f^-(x)=0$, then
$$\frac{\left(D_{\alpha}f\right)(s)}{\left(D_{\alpha}g\right)(s)}\leq\frac{f^+(s)}{g^+(s)},$$
for all $s\in(a,b)$, implying that $(D_{\alpha}f/g)\geq0$. On the other hand if $\lim_{x\downarrow a}g^-(x)=\lim_{x\downarrow a}f^-(x)=0$, then
$$\frac{f^-(t)}{g^-(t)}\leq\frac{\left(D_{\alpha}f\right)(t)}{\left(D_{\alpha}g\right)(t)},$$
for all $t\in(a,b)$, implying that $(D_{\alpha}f/g)\geq0$.
It follows from Theorem \ref{LHR00007} that $f^-/g^-$ is increasing on $(a,b)$.

Finally, if $\left(D_{\alpha}f\right)/\left(D_{\alpha}g\right)$ is strictly increasing on $(a,b)$, then it follows from the first part of the proof that $f^-/g^-$ is increasing on $(a,b)$. Now if $f^-/g^-$ is not strictly increasing, then $f^-/g^-=C$ (constant) on some open interval subset $J$ of $(a,b)$, which implies that $\left(D_{\alpha}f\right)/\left(D_{\alpha}g\right)=C$ on $J$. This contradiction completes the proof of the second statement.
\end{proof}

\begin{remark}\
\begin{enumerate}[leftmargin=.65cm]
\item If $\left(D_{\alpha}f\right)/\left(D_{\alpha}g\right)$ is strictly decreasing (respectively decreasing) on $(a,b)$, then $f^-/g^-$ strictly decreasing (respectively decreasing) on $(a,b)$.
\item If $f$, $g$, and $\alpha$ satisfy the assumptions of Theorem \ref{LHR00023} but either $\lim_{x\uparrow b}g^-(x)=B$ and $\lim_{x\uparrow b}f^-(x)=A$ or $\lim_{x\downarrow a}g^-(x)=B$ and $\lim_{x\downarrow a}f^-(x)=A$ where $A$ and $B$ are real numbers, then the function 
$$\frac{f^--A}{g^--B}$$
is strictly increasing (respectively increasing) on $(a,b)$ if $(D_{\alpha}f)/(D_{\alpha}g)$ is strictly increasing (respectively increasing) on $(a,b)$. This can be shown by applying Theorem \ref{LHR00023} to the functions $f_0=f-A$ and $g_0=g-B$.
\end{enumerate}
\end{remark}

\section{Applications}\label{LHR00024}

\subsection{L'Hospital's Rule via Right-Hand Derivatives}\label{LHR00025}

In this subsection, we explore an interesting special case of Theorem \ref{LHR00018}. It is well-known that L'Hospital's rule works for one-sided derivatives, see Theorems 1, 2, and 3 in Nester and Vyborny \cite{Vyborny1989}. Theorem 1 states that if $f$ and $g$ are continuous functions on the interval $(a,b)$, where $-\infty\leq a<b\leq\infty$, $g$ is monotonic on $(a,b)$, $f(x)$ and $g(x)$ tend to zero as $x\uparrow b$, and $f'_+(x)/g'_+(x)$ tends to $L$ as $x\uparrow b$, then $f(x)/g(x)$ tends to $L$ as $x\uparrow b$.
This theorem assumes that $g'_+(x)\neq 0$ at least near $b$. Also, it works when $g(x)$ tends to $\pm\infty$ as $x\uparrow b$. In addition, a similar version works if we consider left derivatives instead of right derivatives. A potential failure of this theorem happens when we keep getting zeros of the left derivative and the right derivative as $x\uparrow b$. Theorem \ref{LHR00018} offers us an alternative if we assume that at each $x\in(a,b)$, at least one of $g'_-(x)$ and $g'_+(x)$ is not zero. This is because $D_{\alpha}g$ averages $g'_-$ and $g'_+$ when $\alpha(t)=t$, as we mentioned in the fourth statement in Section \ref{qw682jhc42l,k529b}.

To illustrate this we consider an example. Let
$$g(x)=\begin{cases}
A_0+x^2(x-2)^2, & 0<x\leq1\\
A_1+x^4, & 1<x\leq2\\
A_2+(x-2)^2(x-4)^2, & 2<x\leq3\\
A_3+x^4, & 3<x\leq4\\
A_4+(x-4)^2(x-6)^2, & 4<x\leq5\\
\vdots
\end{cases}$$
where $A_0=A_1=0$ and $A_2,A_3,\dots$ are the unique numbers that make $g$ continuous on $(0,\infty)$. That is, $A_{2n}=A_{2n-1}+(2n)^4$ and $A_{2n+1}=A_{2n}+1-(2n+1)^4$, for each natural number $n$. Our goal is to check whether we can use a version of L'Hospital's rule to determine the behavior of $f(x)/g(x)$ as $x\uparrow\infty$, where $f(x)=g(x+1)$. Note that both $f$ and $g$ tend to $\infty$ as $x\uparrow\infty$. It is straightforward to verify that $g'_-(n)=4n^3$ and $g'_+(n)=0$ when $n$ is even, $g'_-(n)=0$ and $g'_+(n)=4n^3$ when $n$ is odd, $g'(x)=4x^3$ when $2n+1<x<2n+2$, and $g'(x)=4(x-2n)(x-2n-1)(x-2n-2)$ when $2n<x<2n+1$. Since as we head to $\infty$ we keep getting zero derivatives either from left or right, Theorem $1$ in \cite{Vyborny1989} does not work. Now with $\alpha(x)=x$ we have $\left(D_{\alpha}g\right)(n)=2n^3$ for all $n$, that is, the average of $g'_-(n)$ and $g'_+(n)$. Then it easy to verify that $\left(D_{\alpha}f\right)(x)/\left(D_{\alpha}g\right)(x)$ tends to $1$ as $x\uparrow\infty$. Consequently, Theorem \ref{LHR00017} tells us that $f(x)/g(x)$ tends to $1$ as $x\uparrow\infty$.

\subsection{An Example Involving Non-Differentiability Near a Limiting Point}\label{LHR00026}

L'Hospital's rule \eqref{LHR00001} fails if there is a sequence of points $s_n$ approaching the limiting point at which $f'(s_n)$ or $g'(s_n)$ does not exist. In this subsection we provide an example to show that Theorem \ref{LHR00017} can be used in such cases.

Consider the functions
$$\alpha(x)=\begin{cases}
\sqrt[3]{x-1}, & 0<x\leq 2\\
2+\sqrt[3]{x-3}, & 2<x\leq 4\\
4+\sqrt[3]{x-5}, & 4<x\leq 6\\
6+\sqrt[3]{x-7}, & 6<x\leq 8\\
\vdots
\end{cases}
$$
and $f(x)=\sqrt[3]{x}$. Our goal is to check whether L'Hospital's rule is applicable to compute the limit of the quotient $f(x)/\alpha(x)$ as $x\uparrow\infty$. The function $\alpha$ is continuous and strictly increasing on $(0,\infty)$, but $\alpha'$ does not exist precisely at each odd number. It is straightforward to verify that if $x\in(2n,2n+2]$, where $n$ is a non-negative integer, then
$$\left(D_{\alpha}f\right)(x)=\sqrt[3]{\left(1-\frac{2n+1}{x}\right)^2}.$$
Since $\left(D_{\alpha}f\right)(x)$ tends to $0$ as $x\uparrow\infty$, Theorem \ref{LHR00017} tells us that $f(x)/\alpha(x)$ tends to $0$ as $x\uparrow\infty$.

\subsection{Discrete and Continuous Versions of L'Hospital's Rule}\label{LHR00027}

The classical L'Hospital's rule is a special case of Theorem \ref{LHR00017}, which can be shown directly by considering $\alpha(x)=x$. On the other hand, the discrete version of L'Hospital's rule, known as the Stolz-Cesaro theorem, is also a special case of Theorem \ref{LHR00017}. To see this, we will state the Stolz-Cesaro theorem and then prove it using Theorem \ref{LHR00017}.

\begin{thm}[Stolz–Cesaro Theorem]
Suppose that $f_n$ is a sequence of real numbers and that $g_n$ is a sequences of positive real numbers such that
$$\lim_{n\uparrow\infty}\frac{f_n}{g_n}=A,$$
where $A\in[-\infty,\infty]$. Let 
$$F_n=\sum_{j=1}^nf_j \text{  and  } G_n=\sum_{j=1}^n g_n.$$
If $\lim_{n\uparrow\infty}G_n=\infty$, then
$$\lim_{n\uparrow\infty}\frac{F_n}{G_n}=A.$$
\end{thm}

\begin{proof}
For each nonnegative integer $m$ we define $F$ on the interval $[m,m+1]$ by
$$F(x)=f_{m+2}\;x+\sum_{j=1}^{m+1}f_j-mf_{m+2},$$
i.e., $F$ is the line segment joining the points
$$\left(m,\sum_{j=1}^{m+1}f_j\right) \text{  and  }\left (m+1,\sum_{j=1}^{m+2}f_j\right).$$
We define $G$ in a similar fashion using the sequence $g_n$. Obviously $F$ and $G$ are continuous on $[0,\infty)$. By choosing $\alpha(x)=x$, it is straightforward to check that $\left(D_{\alpha}F\right)(x)=f_{m+2}$ on $(m,m+1)$ and $\left(D_{\alpha}F\right)(m+1)=(f_{m+2}+f_{m+3})/2$. Similarly, $\left(D_{\alpha}G\right)(x)=g_{m+2}$ on $(m,m+1)$ and $\left(D_{\alpha}G\right)(m+1)=(g_{m+2}+g_{m+3})/2$. Since $f_m/g_m$ tends to $A$ as $m\uparrow\infty$ the sequence $\left(f_{m+2}+f_{m+3}\right)/\left(g_{m+2}+g_{m+3}\right)$ tends to $A$ as $m\uparrow\infty$. Therefore $\left(D_{\alpha}F\right)(x)/\left(D_{\alpha}G\right)(x)$ tends to $A$ as $x\uparrow\infty$. It follows from Theorem \ref{LHR00017} that $F(x)/G(x)$ tends to $A$ as $x\uparrow\infty$, which implies that $F_n/G_n$ tends to $A$ as $n\uparrow\infty$.
\end{proof}

\subsection{An Example Involving Discontinuities}\label{LHR00028}

For the functions
$$f(x)=\begin{cases}
x, & 0<x\leq 1\\
x+1+1/2, & 1<x\leq 2\\
x+2+1/2+1/3, & 2<x\leq 3\\
x+3+1/2+1/3+1/4, & 3<x\leq 4\\
\vdots
\end{cases}\text{  and  } \alpha(x)=\begin{cases}
x, & 0<x\leq 1\\
x+1, &1<x\leq 2\\
x+2, &2<x\leq 3\\
x+3, &3<x\leq 4\\
\vdots
\end{cases}
$$
the goal is to check whether L'Hospital's rule is applicable for the quotient $f(x)/\alpha(x)$ as $x\uparrow\infty$. Observe that $\alpha(x)$ tends to $\infty$ as $x\uparrow\infty$. In addition, the function $\alpha$ is strictly increasing on $(0,\infty)$. Obviously, $\left(D_{\alpha}f\right)(n)=1+1/(n+1)$ for $n=1,2,\dots$ and $\left(D_{\alpha}f\right)(x)=1$ otherwise. Note that $D_{\alpha}\alpha=1$. Since
$\left(D_{\alpha}f\right)(x)$ tends to $1$ as $x\uparrow\infty$, it follows from Theorem \ref{LHR00017} that $f(x)/\alpha(x)$ tends to $1$ as $x\uparrow\infty$.

\section{L'Hospital-Type Rules via Lebesgue-Stieltjes Integrals}\label{LHR00029}

The goal of this section is to present measure-theoretic versions of L'Hospital's rule and L'Hospital's monotone rule. Before we begin, let us review some key facts and notations from measure theory.

Suppose that $-\infty\leq a<b\leq\infty$. The Borel $\sigma$-algebra on $(a,b)$, denoted by $\mathbb{B}((a,b))$, is the $\sigma$-algebra generated by all open interval subsets of $(a,b)$. A real-valued function $h$ of locally bounded variation on $(a,b)$ induces a Lebesgue-Stieltjes measure on $(s,t)$, for each $a<s<t<b$, which we denote by $dh$. The measure of $(s,t)$ is $dh\left((s,t)\right)=h^-(t)-h^+(s)$. For any $S\subset(a,b)$, the measure of $S$ is
$$dh(S)=\inf\left\{\sum_{j\in J}\left(h^-(t_j)-h^+(s_j)\right):S\subset\bigcup_{j\in J}(s_j,t_j), J\text{ is countable}\right\}.$$
Recall that $dh$ is countably additive on $\mathbb{B}((s,t))$, for any $a<s<t<b$.

If $\alpha$ is increasing on $(a,b)$, then $d\alpha$ is a positive measure. In this case, $d\alpha$ is countably additive on $\mathbb{B}((a,b))$. A measurable function $f:(a,b)\to\mathbb{R}$ is said to be integrable with respect to $d\alpha$ if
$$\int_{(a,b)}|f|d\alpha<\infty.$$
The function $f$ is said to be locally integrable on $(a,b)$ if it is integrable on every interval $(s,t)$, where $a<s<t<b$. Also, it is integrable near $b$ if it is integrable on $(x_0,b)$ for some $x_0\in(a,b)$.

Suppose that $\alpha$ is increasing on $(a,b)$ and $h$ is of locally bounded variation on $(a,b)$. We say that $dh$ is locally absolutely continuous with respect to $d\alpha$ if $d\alpha(S)=0$ implies that $dh(S)=0$ as long as $S\in\mathbb{B}\left((s,t)\right)$ for some $a<s<t<b$.

Before moving on, we introduce the following application to the mean value theorems introduced in Section \ref{LHR00005}.

\begin{thm}\label{LHR00029.01}
Suppose that $f$ and $\alpha$ satisfy assumptions $1$, $2$, $4$, and $5$ in Hypothesis \ref{LHR00006} and that $D_{\alpha}f$ is locally bounded on $(a,b)$. The following statements hold:
\begin{enumerate}[leftmargin=0.65cm]
\item If $s<t$ in $(a,b)$, then there exists $M>0$ such that for all $x<y$ in $[s,t]$ we have $|f^-(y)-f^+(x)|\leq M(\alpha^-(y)-\alpha^+(x))$.
\item $f^-$ is of locally bounded variation on $(a,b)$.
\item $df$ is locally absolutely continuous with respect to $d\alpha$ on $(a,b)$.
\end{enumerate} 
\end{thm}

\begin{proof}
We will prove only the first claim, as the last two claims follow from it. Fix $s<t$ in $(a,b)$. Since $D_{\alpha}f$ is locally bounded on $(a,b)$, there exists $M>0$ such that $|D_{\alpha}f|\leq M$ on $[s,t]$. Let $x<y$ in $[s,t]$ be given. Then the first statement in Corollary \ref{LHR00016} informs us that there are $u$ and $v$ in $(x,y)$ such that
$$\left(D_{\alpha}f\right)(u)\leq\frac{f^-(y)-f^+(x)}{\alpha^-(y)-\alpha^+(x)}\leq\left(D_{\alpha}f\right)(v).$$
But $-M\leq\left(D_{\alpha}f\right)(u)\leq\left(D_{\alpha}f\right)(v)\leq M$, so $|f^-(y)-f^+(x)|\leq M(\alpha^-(y)-\alpha^+(x))$, which completes the proof.
\end{proof}

The following result is a special case of the differentiation theorem for Radon measures, see Theorem 1.30 in Evans and Gariepy \cite{MR3409135}.

\begin{thm}\label{LHR00029.5}
Suppose that $\alpha$ is increasing on $(a,b)$ and that $h$ is of locally bounded variation on $(a,b)$.
\begin{enumerate}[leftmargin=.65cm]
\item The derivative $D_{\alpha}h$ exists almost everywhere on $(a,b)$ with respect to $d\alpha$.
\item If $dh$ is locally absolutely continuous with respect to $d\alpha$, then
$D_{\alpha}h$ is locally integrable on $(a,b)$ with respect to $d\alpha$ and
$$h^-(t)-h^+(s)=\int_{(s,t)}\left(D_{\alpha}h\right)d\alpha$$
for any $s<t$ in $(a,b)$.
\end{enumerate}
\end{thm}

Now we are ready to present a Lebesgue-Stieltjes versions of L'Hospital's rule and L'Hospital's monotone rule.

\begin{thm}[L'Hospital's Rule via Lebesgue-Stieltjes Integrals]\label{LHR00030}
Suppose that $-\infty\leq a<b\leq\infty$ and let $\alpha$ be a strictly increasing function on $(a,b)$. Let $u$ and $v$ be real-valued functions on $(a,b)$ that are locally integrable on $(a,b)$ with respect to $d\alpha$ such that
either $v>0$ almost everywhere on $(a,b)$ with respect to $d\alpha$ or $v<0$ almost everywhere on $(a,b)$ with respect to $d\alpha$. In addition, suppose that $w:(a,b)\to\mathbb{R}$ satisfies $u/v=w$ almost everywhere with respect to $d\alpha$ and 
$$\lim_{x\uparrow b}w(x)=A,$$
where $A\in[-\infty,\infty]$. The following statements hold:
\begin{enumerate}[leftmargin=.65cm]
\item If $u$ and $v$ are integrable near $b$ with respect to $d\alpha$, then
$$\lim_{x\uparrow b}\frac{\int_{(x,b)}ud\alpha}{\int_{(x,b)}vd\alpha}=A.$$
\item If either $\int_r^bvd\alpha=\infty$ or $\int_r^bvd\alpha=-\infty$, where $r$ is any point in $(a,b)$, then
$$\lim_{x\uparrow b}\frac{\int_{(r,x)}ud\alpha}{\int_{(r,x)}vd\alpha}=A.$$
\end{enumerate}
\end{thm}

\begin{proof}
We will give a proof for the case $A\in(-\infty,\infty)$. The cases $A=\mp\infty$ can be treated similarly. Let $\epsilon>0$ be given. Then there exists $x_0\in(a,b)$ such that $|w-A|<\epsilon$ on $(x_0,b)$. It follows that
\begin{equation}\label{LHR00031}
|u-Av|<\epsilon|v|
\end{equation}
almost everywhere on $(x_0,b)$ with respect to $d\alpha$. Hence for any $x\in(x_0,b)$ we have
$$\left|\int_{(x,b)}ud\alpha-A\int_{(x,b)}vd\alpha\right|\leq\int_{(x,b)}|u-Av|d\alpha\leq\epsilon\int_{(x,b)}|v|d\alpha=\epsilon\left|\int_{(x,b)}vd\alpha\right|,$$
where the last equality is justified by the fact that either $v$ is positive or negative almost everywhere on $(a,b)$ with respect to $d\alpha$. This completes the proof of the first statement.

To prove the second statement note that inequality \eqref{LHR00031} gives
$$\left|\int_{(r,x)}ud\alpha-A\int_{(r,x)}vd\alpha\right|\leq\epsilon\left|\int_{(r,x)}vd\alpha\right|,$$
whenever  $x_0\leq r<x<b$. Thus for any $r\geq x_0$ we have
$$\lim_{x\uparrow b}\frac{\int_{(r,x)}ud\alpha}{\int_{(r,x)}vd\alpha}=A.$$
Finally, using additivity of integrals and the fact that $\int_r^bvd\alpha=\mp\infty$, it is easy to show that the last inequality holds when $r<x_0$.
\end{proof}

\begin{thm}[L'Hospital's Monotone Rule via Lebesgue-Stieltjes Integrals]\label{LHR00032}
Suppose that $-\infty\leq a<b\leq\infty$ and let $\alpha$ be a strictly increasing function on $(a,b)$. Let $u$ and $v$ be real-valued functions on $(a,b)$ that are integrable on $(a,t)$ with respect to $d\alpha$ for every $t\in(a,b)$ such that either $v>0$ almost everywhere on $(a,b)$ with respect to $d\alpha$ or $v<0$ almost everywhere on $(a,b)$ with respect to $d\alpha$. Define $h:(a,b)\to\mathbb{R}$ by
$$h(x)=\frac{\int_{(a,x)}ud\alpha}{\int_{(a,x)}vd\alpha}.$$
The following statements hold:
\begin{enumerate}[leftmargin=.65cm]
\item If there exists an increasing function $w:(a,b)\to\mathbb{R}$ such that $w=u/v$ almost everywhere on $(a,b)$ with respect to $d\alpha$, then $h$ is increasing on $(a,b)$.
\item If there exists a strictly increasing function $w:(a,b)\to\mathbb{R}$ such that $w=u/v$ almost everywhere on $(a,b)$ with respect to $d\alpha$, then $h$ is strictly increasing on $(a,b)$.
\end{enumerate}
\end{thm}

\begin{proof}
To prove the first claim, note that the equation
$$h(t)=h^+(s)+\int_{(s,t)}(D_{\alpha}h)d\alpha$$
holds for any $s<t$ in $(a,b)$. Hence it suffices to show that $D_{\alpha}h\geq0$ almost everywhere on $(a,b)$ with respect to $d\alpha$.

Define the left-continuous functions $\phi:(a,b)\to\mathbb{R}$ and $\psi:(a,b)\to\mathbb{R}$ by
$$\phi(x)=\int_{(a,x)}ud\alpha\text{ and }\psi(x)=\int_{(a,x)}vd\alpha.$$
Then
$$D_{\alpha}h=\frac{u\psi-v\phi}{\psi\psi^+}=\frac{v}{\psi^+}\left(w-\frac{\phi}{\psi}\right)
$$
almost everywhere on $(a,b)$ with respect to $d\alpha$. Since $v/\psi^+>0$ almost everywhere on $(a,b)$ with respect to $d\alpha$, completing the proof requires showing that 
$$\frac{\phi(x)}{\psi(x)}\leq w(x)$$
for all $x\in(a,b)$.

Fix $x\in(a,b)$ and let $r\in(a,x)$ be given. Define $\gamma:(r,x)\to\mathbb{R}$ by
$$\gamma=\left(\phi(x)-\phi^+(r)\right)\psi-\left(\psi(x)-\psi^+(r)\right)\phi.$$
Observe that $\gamma^+(r)=\gamma(x)$ and that
$$
\begin{aligned}
D_{\alpha}\gamma
&
=\left(\phi(x)-\phi^+(r)\right)v-\left(\psi(x)-\psi^+(r)\right)u
\\
&
=v\left(\psi(x)-\psi^+(r)\right)\left(\frac{\phi(x)-\phi^+(r)}{\psi(x)-\psi^+(r)}-w\right)
\end{aligned}
$$
almost everywhere on $(r,x)$ with respect to $d\alpha$. Then
$$0=\gamma(x)-\gamma^+(r)=\int_{(r,x)}(D_{\alpha}\gamma)d\alpha.$$
Therefore the set
$$\left\{s\in(r,x):(D_{\alpha}\gamma)(s)\leq0\right\}=\left\{s\in(r,x):\frac{\phi(x)-\phi^+(r)}{\psi(x)-\psi^+(r)}\leq w(s)\right\}$$
has nonzero $d\alpha$-measure. Note that the equality of the above two sets is justified by that $v\left(\psi(x)-\psi^+(r)\right)>0$ almost everywhere on $(r,x)$ with respect to $d\alpha$. Since $w$ is increasing on $(a,b)$, we conclude that
$$\frac{\phi(x)-\phi^+(r)}{\psi(x)-\psi^+(r)}\leq w(x).$$
Finally, since $r\in(a,x)$ is arbitrary in the last inequality, letting $r\downarrow a$ gives
$$\frac{\phi(x)}{\psi(x)}\leq w(x),$$
which completes the proof of the first claim.

To prove the second claim suppose that $w$ is strictly increasing on $(a,b)$. It follows from the first statement that $h$ is increasing on $(a,b)$. If $h$ is not strictly increasing on $(a,b)$, then there exist $C\in\mathbb{R}$ such that $h=C$ on some interval $(s,t)$, where $a<s<t<b$. Then $u=Cv$ almost everywhere on $(s,t)$. Hence $w=C$ almost everywhere on $(s,t)$, contradicting that $w$ is strictly increasing.
\end{proof}

\section{Comparison of L'Hospital-Type Rules}\label{LHR00037.5}

The goal of this section is to show that Theorem \ref{LHR00030} is more general than Theorem \ref{LHR00017} and that Theorem \ref{LHR00032} is more general than Theorem \ref{LHR00023}. We will prove only the first claim, as a similar argument can be used to prove the second claim. To achieve this, we introduce the following result.

\begin{thm}\label{LHR00038}
Suppose that $-\infty\leq a<b\leq\infty$.
If $\alpha$ is a strictly increasing function on $(a,b)$ and $f$ is an increasing function on $(a,b)$ such that $\left(D_{\alpha}f\right)(x)$ exists for all $x\in(a,b)$, then $df$ is locally absolutely continuous with respect to $d\alpha$.
\end{thm}

\begin{proof}
Fix an interval $(s,t)$ that satisfies $a<s<t<b$. Suppose that $S\subset(s,t)$ is a Borel set for which $d\alpha(S)=0$. For each $n\in\mathbb{N}$ let $A_n=\{x\in S:\left(D_{\alpha}f\right)(x)<n\}$. Then $S=\bigcup_{n=1}^{\infty}A_n$, which implies $df(S)\leq\sum_{n=1}^{\infty} df(A_n)$. Thus it suffices to show that $df(A_n)=0$ for every $n\in\mathbb{N}$. 

Fix $n\in\mathbb{N}$ and let $\epsilon>0$ be given. Since $d\alpha(A_n)=0$, there exists an open set $O\subset(s,t)$ such that $O\supset A_n$ and $d\alpha(O)<\epsilon/n$. Formula \eqref{LHR00004.5}, together with the inequality $\left(D_{\alpha}f\right)(x)<n$ for all $x\in A_n$, inform us that for each $x\in A_n$, there exists $\delta_x>0$ such that $(x-\delta_x,x+\delta_x)\subset O$ and
$$\frac{df\left([x-h,x+h]\right)}{d\alpha\left([x-h,x+h]\right)}\leq n$$
whenever $h\in(0,\delta_x)$.
The collection $\mathcal{F}=\{[x-h,x+h]:x\in A_n, h\in(0,\delta_x)\}$, together with the set $A_n$ satisfies the hypothesis of Theorem $1.28$ in \cite{MR3409135}. Then there exists a countable collection of pairwise disjoint intervals $\{[x_j-h_j,x_j+h_j]:j\in J\}$ in $\mathcal{F}$ such that
$$df\left(A_n\backslash\bigcup_{j\in J}[x_j-h_j,x_j+h_j]\right)=0.$$
Hence
$$
df\left(A_n\right)\leq df\left(\bigcup_{j\in J}[x_j-h_j,x_j+h_j]\right)=\sum_{j\in J}df\left([x_j-h_j,x_j+h_j]\right)\leq n d\alpha(O)<\epsilon.
$$
We conclude that $(df)^*(A_n)=0$, which completes the proof.
\end{proof}

Now suppose that the functions $f$, $g$, and $\alpha$ satisfy the requirements of Theorem \ref{LHR00017}. Our goal is to show that the conclusion of Theorem \ref{LHR00017}, i.e.,
$$\lim_{x\uparrow b}\frac{f^-(x)}{g^-(x)}=A,$$
can be achieved using Theorem \ref{LHR00030}.

We consider only the case  $\lim_{x\uparrow b}f^-(x)=\lim_{x\uparrow b}g^-(x)=0$ and $A\in(-\infty,\infty)$, as similar proofs can be constructed for the remaining cases. First, set $u=D_{\alpha}f$ and $v=D_{\alpha}g$. We claim that $u$ and $v$ are integrable near $b$ with respect to $d\alpha$ and that
$$-g^+(s)=\int_{(s,b)}vd\alpha \text{  and  }-f^+(s)=\int_{(s,b)}ud\alpha$$
when $s$ is sufficiently close to $b$.  Establishing this claim will allow us to use the second statement of Theorem \ref{LHR00030} as desired.

To establish this claim note that since $v$ has the same sign on $(a,b)$, it follows from Theorems \ref{LHR00003} and \ref{LHR00007} that either $g^-$ is strictly increasing on $(a,b)$ or $g^-$ is strictly decreasing on $(a,b)$. Then Theorem \ref{LHR00038} informs us that $dg$ is locally absolutely continuous with respect to $d\alpha$, i.e.,
$$g^-(t)-g^+(s)=\int_{(s,t)}vd\alpha$$
whenever $a<s<t<b$. By letting $t\uparrow b$ we obtain
$$-g^+(s)=\int_{(s,b)}vd\alpha,$$
for all $a<s<b$, justified by the monotone convergence theorem.

Now since 
$$\lim_{x\uparrow b}\frac{u(x)}{v(x)}=A$$
there exists $r\in(a,b)$ such that $|u(x)-Av(x)|<|v(x)|$
for all $x\in(r,b)$. Hence $|u(x)|<(1+|A|)|v(x)|$ for all $x\in(r,b)$. Since $v$ is integrable near $b$ with respect to $d\alpha$, so is $u$.

If $A\neq0$, then there exists $r_0\in(r,b)$ such that $u$ and $v$ have the same sign on $(r_0,b)$. This implies that $u$ has the same sign on $(r,b)$. Hence, a similar argument to the one used for $v$ shows that
$$-f^+(s)=\int_{(s,b)}ud\alpha,$$
for all $r_0<s<b$.

If $A=0$, then 
$$\lim_{x\uparrow b}\frac{u(x)+v(x)}{v(x)}=1.$$
By applying a similar argument to the one used for $v$, there exists $r_1\in(r,b)$ such that
$$-f^+(s)-g^+(s)=\int_{(s,b)}(u+v)d\alpha$$
for all $r_1<s<b$. Hence 
$$-f^+(s)=\int_{(s,b)}ud\alpha$$
for all $r_1<s<b$.


\begin{thebibliography}{10}

\bibitem{MR1227499}
G.~D. Anderson, M.~K. Vamanamurthy, and M.~Vuorinen.
\newblock Inequalities for quasiconformal mappings in space.
\newblock {\em Pacific J. Math.}, 160(1):1--18, 1993.

\bibitem{balogh2016functional}
Zolt{\'a}n~M Balogh, Orif~O Ibrogimov, and Boris~S Mityagin.
\newblock Functional equations and the cauchy mean value theorem.
\newblock {\em Aequationes mathematicae}, 90:683--697, 2016.

\bibitem{doi:10.1080/00029890.1986.11971912}
R.~P. Boas.
\newblock Counterexamples to l'hôpital's rule.
\newblock {\em The American Mathematical Monthly}, 93(8):644--645, 1986.

\bibitem{10.2307/2310244}
D.~S. Carter.
\newblock L'hospital's rule for complex-valued functions.
\newblock {\em The American Mathematical Monthly}, 65(4):264--266, 1958.

\bibitem{Estrada2017LHpitalsMR}
Ricardo Estrada and Miroslav Pavlovic.
\newblock L’h{\^o}pital’s monotone rule, gromov’s theorem, and operations
  that preserve the monotonicity of quotients.
\newblock {\em Publications De L'institut Mathematique}, 101:11--24, 2017.

\bibitem{MR3409135}
Lawrence~C. Evans and Ronald~F. Gariepy.
\newblock {\em Measure theory and fine properties of functions}.
\newblock Textbooks in Mathematics. CRC Press, Boca Raton, FL, revised edition,
  2015.

\bibitem{10.2307/2323839}
Gianluca Gorni.
\newblock A geometric approach to l'h{\^o}pital's rule.
\newblock {\em The American Mathematical Monthly}, 97(6):518--523, 1990.

\bibitem{Ivlev2014OnAG}
V.~V. Ivlev and I.~A. Shilin.
\newblock On a generalization of l'hopital's rule for multivariable functions.
\newblock {\em arXiv: History and Overview}, 2014.

\bibitem{lakshmikantham2017monotone}
V.~Lakshmikantham.
\newblock {\em Monotone Iterative Techniques for Discontinuous Nonlinear
  Differential Equations}.
\newblock CRC Press, 2017.

\bibitem{doi:10.1080/00029890.2020.1793635}
Gary~R. Lawlor.
\newblock l’hôpital’s rule for multivariable functions.
\newblock {\em The American Mathematical Monthly}, 127(8):717--725, 2020.

\bibitem{10.2307/2040953}
Cheng-Ming Lee.
\newblock Generalizations of l'hôpital's rule.
\newblock {\em Proceedings of the American Mathematical Society},
  66(2):315--320, 1977.

\bibitem{lozada2020some}
German Lozada-Cruz.
\newblock Some variants of cauchy's mean value theorem.
\newblock {\em International Journal of Mathematical Education in Science and
  Technology}, 51(7):1155--1163, 2020.

\bibitem{Mateljevic2013GeneralizationsOT}
M.~Mateljevic, Marek Světl{\'i}k, Miloljub Albijanic, and Neboj{\"y}sa Savic.
\newblock Generalizations of the lagrange mean value theorem and applications.
\newblock {\em Filomat}, 27:515--528, 2013.

\bibitem{10.2307/24896380}
Miodrag Mateljević, Marek Svetlik, Miloljub Albijanić, and Nebojša Savić.
\newblock Generalizations of the lagrange mean value theorem and applications.
\newblock {\em Filomat}, 27(4):515--528, 2013.

\bibitem{Matkowski+2010+765+774}
Janusz Matkowski.
\newblock Generalizations of lagrange and cauchy mean-value theorems.
\newblock {\em Demonstratio Mathematica}, 43(4):765--774, 2010.

\bibitem{10.2307/2318210}
A.~M. Ostrowski.
\newblock Note on the bernoulli-l'hospital rule.
\newblock {\em The American Mathematical Monthly}, 83(4):239--242, 1976.

\bibitem{MR1888920}
Iosif Pinelis.
\newblock L'{H}ospital type results for monotonicity, with applications.
\newblock {\em JIPAM. J. Inequal. Pure Appl. Math.}, 3(1):Article 5, 5, 2002.

\bibitem{rudin1976principles}
W.~Rudin.
\newblock {\em Principles of Mathematical Analysis}.
\newblock International series in pure and applied mathematics. McGraw-Hill,
  1976.

\bibitem{doi:10.1142/3857}
P~K Sahoo and T~Riedel.
\newblock {\em Mean Value Theorems and Functional Equations}.
\newblock WORLD SCIENTIFIC, 1998.

\bibitem{10.2307/2307183}
A.~E. Taylor.
\newblock L'hospital's rule.
\newblock {\em The American Mathematical Monthly}, 59(1):20--24, 1952.

\bibitem{10.2307/44152304}
Marco Vianello.
\newblock {A GENERALIZATION OF L’HÔPITAL’S RULE VIA ABSOLUTE CONTINUITY
  AND BANACH MODULES}.
\newblock {\em Real Analysis Exchange}, 18(2):557 -- 567, 1993.

\bibitem{Vyborny1989}
Nester~R. Vyborny, R.
\newblock L'hôpital's rule, a counterexample.
\newblock {\em Elemente der Mathematik}, 44(5):116--121, 1989.

\bibitem{Witula2012MeanvalueTF}
Roman Witula, Edyta Hetmaniok, and Damian Słota.
\newblock Mean-value theorems for one-sided differentiable functions.
\newblock 2012.

\end{thebibliography}
\end{document}